\documentclass[sigconf]{acmart}

\usepackage{booktabs} 
\usepackage{paralist}
\usepackage[ruled]{algorithm2e}
\usepackage{microtype,amsmath,amsthm,amssymb,enumitem,multirow,algorithmic,mathtools}
\usepackage{graphicx}
\usepackage{subcaption}
\usepackage{booktabs}
\newtheorem{theorem}{Theorem}[section]
\newtheorem{lemma}[theorem]{Lemma}

\SetKwComment{Comment}{$\triangleright$\ }{}

\usepackage{hyperref}

\newcommand{\bw}{{\boldsymbol w}}
\newcommand{\bx}{{\boldsymbol x}}
\newcommand{\bs}{{\boldsymbol s}}
\newcommand{\by}{{\boldsymbol y}}

\newcommand{\bp}{{\boldsymbol p}}
\newcommand{\bd}{{\boldsymbol d}}

\newcommand{\bu}{\mathbf{u}}

\newcommand{\R}{\mathbb{R}}

\newcommand{\diag}{\text{diag}}

\newcommand{\TheTitle}{A Distributed Quasi-Newton Algorithm for
Empirical Risk Minimization with Nonsmooth Regularization}
\newcommand{\RunningTitle}{Distributed Proximal LBFGS}

\begin{document}
\title[\RunningTitle]{{\TheTitle}}
\author{Ching-pei Lee}
\affiliation{%
	\institution{Department of Computer Sciences\\ University
	of Wisconsin-Madison}
	\streetaddress{1210 W. Dayton St}
	\city{Madison}
	\state{Wisconsin}
	\postcode{53706}
}
\email{ching-pei@cs.wisc.edu}

\author{Cong Han Lim}
\affiliation{%
	\institution{Wisconsin Institute for Discovery\\ University
	of Wisconsin-Madison}
	\streetaddress{330 N. Orchard St}
	\city{Madison}
	\state{Wisconsin}
	\postcode{53715}
}
\email{clim9@wisc.edu}

\author{Stephen~J. Wright}
\affiliation{%
	\institution{Department of Computer Sciences\\ University
	of Wisconsin-Madison}
	\streetaddress{1210 W. Dayton St}
	\city{Madison}
	\state{Wisconsin}
	\postcode{53706}
}
\email{swright@cs.wisc.edu}

\begin{abstract}

We propose a communication- and computation-efficient distributed
optimization algorithm using second-order information for solving ERM
problems with a nonsmooth regularization term.
Current second-order and quasi-Newton methods for this
problem either do not work well in the distributed setting or
work only for specific regularizers. Our algorithm uses successive
quadratic approximations, and we describe how to maintain an
approximation of the Hessian and solve subproblems efficiently in a
distributed manner.
The proposed method enjoys global linear convergence for
a broad range of non-strongly convex problems that includes the most
commonly used ERMs, thus requiring lower communication complexity.
It also converges on non-convex problems, so has the
potential to be used on applications such as deep learning.
Initial computational results on convex problems demonstrate that our
method significantly improves on communication cost and running time
over the current state-of-the-art methods.

\end{abstract}

\copyrightyear{2018}
\acmYear{2018}
\setcopyright{acmlicensed}
\acmConference[KDD '18]{The 24th ACM SIGKDD International
Conference on Knowledge Discovery \& Data Mining}{August
19--23,2018}{London, United Kingdom}
\acmBooktitle{KDD '18: The 24th ACM SIGKDD International
	Conference on Knowledge Discovery \& Data Mining, August 19--23,
2018, London, United Kingdom}
\acmPrice{15.00}
\acmDOI{10.1145/3219819.3220075}
\acmISBN{978-1-4503-5552-0/18/08}


\maketitle
\section{Introduction}
\label{sec:intro}
We consider solving the following regularized problem in a distributed
manner:
\begin{equation}
\min_{\bw \in \R^d} \quad F(\bw)  \coloneqq
	f(X^T \bw) + g(\bw),
\label{eq:f}
\end{equation}
where $X$ is a $d$ by $n$ real-valued matrix, $g$ is a convex, closed,
and extended-valued proper function that can be nondifferentiable, and
$f$ is a differentiable function whose gradient is Lipschitz
continuous with parameter $L > 0$. Each column of $X$ represents a
single data point or instance, and we assume that the set of data
points is partitioned and spread across $K$ machines (i.e.\
distributed \emph{instance-wise}). We can write $X$ as
\begin{equation*}
	X \coloneqq \left[X_1, X_2,\dotsc, X_K \right],
\end{equation*}
where $X_k$ is stored exclusively on the $k$th machine. We further
assume that $f$ shares the same block-separable structure and
can be written as follows:
\begin{equation*}
	f\left( X^T \bw \right) = \sum_{k=1}^K f_k \left(X_k^T \bw\right).
\end{equation*}
Unlike our instance-wise setting, some existing works consider the \emph{feature-wise} partition
setting, under which  $X$ is partitioned by rows rather than columns.
Although the feature-wise setting is a simpler one for algorithm
design when $g$ is separable, storage of different features on different
machines is often impractical.


The bottleneck in performing distributed optimization is often the
high cost of communication between machines. For \eqref{eq:f},
the time required to retrieve $X_k$ over a network can greatly exceed
the time needed to compute $f_k$ or its gradient with locally stored
$X_k$.  Moreover, we incur a delay at the beginning of each round of
communication due to the overhead of establishing connections between
machines. This latency prevents many efficient single-core algorithms
such as coordinate descent (CD) and stochastic gradient and their
asynchronous parallel variants from being employed in large-scale
distributed computing setups. Thus, a key aim of algorithm design for
distributed optimization is to improve the communication efficiency
while keeping the computational cost affordable. Batch methods are
preferred in this context, because fewer rounds of communication occur
in distributed batch methods.




When $F$ is differentiable, many efficient batch methods can be used
directly in distributed environments to solve \eqref{eq:f}.
For example, Nesterov's accelerated gradient (AG) \citep{Nes83a}
enjoys low iteration complexity, and since each iteration of AG only
requires one round of communication to compute the new gradient, it
also has good communication complexity.
Although its supporting theory is not particularly strong, the
limited-memory BFGS (LBFGS) method \citep{LiuN89a} is popular among
practitioners of distributed optimization.
It is the default algorithm for solving $\ell_2$-regularized smooth
ERM problems in Apache Spark's distributed machine learning library
\citep{Men16a}, as it is empirically much faster than AG (see, for
example, the experiments in \citet{WanLL16a}). Other modified batch
methods that utilize the Hessian of the objective in various ways are
also communication-efficient under their own additional assumptions
\citep{ShaSZ14a, ZhaL15a, LeeWCL17a,ZhuCJL15a,LinTLL14a}.

When $g$ is nondifferentiable, neither LBFGS nor Newton's method
can be applied directly.  Leveraging curvature information from $f$
can still be beneficial in this setting. For example, the orthant-wise
quasi-Newton method OWLQN \citep{AndG07a} adapts the LBFGS algorithm
to the special nonsmooth case in which $g(\cdot) \equiv \|\cdot\|_1$,
and is popular for distributed optimization of
$\ell_1$-regularized ERM problems.  Extension of this
approach to other nonsmooth $g$ is not well understood, and the
convergence guarantees are only asymptotic, rather than global.
%

To the best of our knowledge, for ERMs with \emph{general} nonsmooth
regularizers in the instance-wise storage setting,
proximal-gradient-like methods \citep{WriNF09a, Nes13a} are the only
practical distributed optimization algorithms with convergence
guarantees. Since these methods barely use the Hessian information of
the smooth part (if at all), we suspect that proper utilization of
second-order information has the potential to improve convergence
speed and therefore communication efficiency dramatically.
We thus propose a practical distributed inexact variable-metric
algorithm for general \eqref{eq:f} which uses gradients and which updates
information from previous iterations to estimate curvature of the
smooth part $f$ in a communication-efficient manner. We describe
construction of this estimate and solution of the corresponding
subproblem. We also provide convergence rate guarantees, which also
bound communication complexity.  These rates improve on existing
distributed methods, even those tailor-made for specific regularizers.


Our algorithm leverages the more general framework provided in
\citet{LeeW18a}, and our major contribution in this work is to describe
how the main steps of the framework can be implemented efficiently in
a distributed environment. Our approach has both good communication
and computational complexity, unlike certain approaches that  focus only
on communication at the expense of computation (and ultimately overall
time).
We believe that this work is the first to propose, analyze, and
implement a practically feasible distributed optimization method for
solving \eqref{eq:f} with general nonsmooth regularizer $g$ under the
instance-wise storage setting.

Our algorithm and implementation details are given in
Section~\ref{sec:alg}.  Communication complexity and the effect of the
subproblem solution inexactness are analyzed in
Section~\ref{sec:analysis}.  Section~\ref{sec:related} discusses
related works, and empirical comparisons are conducted in
Section~\ref{sec:exp}.  Concluding observations appear in
Section~\ref{sec:conclusions}.

\subsection*{Notation}
We use the following notation.
\begin{itemize}[
topsep=0pt,
partopsep=0pt,
leftmargin=*,
itemsep=0pt,
parsep=0pt
]
\item $f(X^T \bw)$ is abbreviated as $\tilde f(\bw)$.
\item $\|\cdot\|$ denotes the 2-norm, both for vectors and for matrices.
\item Given any symmetric positive semi-definite matrix $H \in
	\R^{d\times d}$ and any vector $\bp \in \R^d$,
	$\|\bp\|_H$ denotes the semi-norm $\sqrt{\bp^T H \bp}$.
\end{itemize}

\section{Algorithm}
\label{sec:alg}

At each iteration of our algorithm for optimizing \eqref{eq:f}, we
construct a subproblem that consists of a quadratic approximation of
$\tilde f$ added to the original regularizer $g$.  Specifically, given
the current iterate $\bw$, we choose a positive semi-definite $H$ and
define
\begin{equation*}
Q_H(\bp;\bw) \coloneqq \nabla \tilde f(\bw)^T
\bp + \frac12 {\|\bp\|_H^2}+ g(\bw + \bp) - g(\bw),
\end{equation*}
the update direction is obtained by approximately solving
\begin{equation}
\min_{\bp \in \R^d}\quad Q_H(\bp;\bw).
\label{eq:quadratic}
\end{equation}
A line search procedure determines a suitable stepsize $\alpha$, and
we perform the update $\bw \leftarrow \bw + \alpha \bp$.

We now discuss the following issues in the distributed setting:
communication cost, the computation of $\nabla \tilde f$, the choice
and construction of $H$, procedures for solving \eqref{eq:quadratic},
and the line search procedure.
In our description, we sometimes need to split some $n$-dimensional
vectors over the machines, in accordance with the following disjoint
partition $J_1,\dotsc,J_k$ of $\{1,\dotsc,d\}$:
\begin{equation*}
	J_i \cap J_k = \phi, \forall i \neq k,\quad \cup_{i=1}^K J_i =
	\{1,\dotsc,d\}.
\end{equation*}
\subsection{Communication Cost}
For the ease of description, we assume the {\em allreduce} model of
MPI \citep{MPI94a}, but it is also straightforward to extend the
framework to a master-worker platform.
Under this
model, all machines simultaneously fulfill master and worker roles,
and any data transmitted is broadcast to all machines.
This can be considered as equivalent to conducting one map-reduce
operation and then broadcasting the result to all nodes.  The
communication cost for the allreduce operation on a $d$-dimensional
vector under this model is
\begin{equation}
	\log \left( K \right) T_{\text{initial}} + d T_{\text{byte}},
	\label{eq:commcost}
\end{equation}
where $T_{\text{initial}}$ is the latency to establish connection
between machines, and $T_{\text{byte}}$ is the per byte transmission
time (see, for example, \citet[Section 6.3]{ChaHPV07a}).

The first term in \eqref{eq:commcost} also explains why batch methods
are preferable. Even if methods that frequently update the iterates
communicate the same amount of bytes, it takes more rounds of
communication to transmit the information, and the overhead of
$\log (K) T_{\text{initial}}$ incurred at every round of communication
makes this cost dominant, especially when $K$ is large.


In subsequent discussion, when an allreduce operation is performed on
a vector of dimension $O(d)$, we simply say that $O(d)$ communication
is conducted. We omit the latency term since batch methods like ours
tend to have only a small constant number of rounds of communication
per iteration. By contrast, non-batch methods such as CD
or stochastic gradient require $O(d)$ or $O(n)$ rounds of
communication per epoch and therefore face much more significant
latency issues.

\subsection{Computing $\nabla \tilde f$}
The gradient of $\tilde f$ has the form
\begin{equation}
\nabla \tilde f(\bw) = X \nabla f(X^T \bw) = \sum_{k=1}^K \left(X_k
\nabla f_k (X_k^T \bw)\right).
\label{eq:grad}
\end{equation}
We see that, except for the sum over $k$, the computation can be
conducted locally provided $\bw$ is available to all machines. Our
algorithm maintains $X_k^T \bw$ on the $k$th machine throughout, and
the most costly steps are the matrix-vector multiplications between $X_k$
and $\nabla f_k(X_k^T \bw)$, and $X^T \bw$. The local $d$-dimensional
partial gradients are then aggregated through an allreduce operation.

\subsection{Constructing a good $H$ efficiently}
\label{subsec:products}


We use the Hessian approximation constructed by the LBFGS algorithm
\citep{LiuN89a}, and propose a way to maintain and utilize it
efficiently in a distributed setting. Using the compact representation
in \citet{ByrNS94a}, given a prespecified integer $m > 0$, at the $t$th
iteration for $t > 0$, let $\tilde m \coloneqq \min(m,t)$, and define
\begin{equation*}
\bs_i \coloneqq \bw_{i+1} - \bw_i, \quad
\by_i \coloneqq \nabla \tilde f (\bw_{i+1}) - \nabla \tilde f(\bw_i),
\quad \forall i.
\end{equation*}
The LBFGS Hessian approximation matrix is
\begin{equation}
	H_t = \gamma_t I - U_t M_t^{-1} U_t^T,
\label{eq:Hk}
\end{equation}
where
\begin{subequations}
\label{eq:M}
\begin{align}
U_t &\coloneqq \left[\gamma_t S_t, Y_t\right],\quad
	M_t \coloneqq \left[\begin{array}{cc}
		\gamma_t S_t^T S_t, & L_t\\
		L_t^T & - D_t \end{array}\right],\\
\gamma_t &\coloneqq \frac{\bs_{t-1}^T \bs_{t-1}}{\bs_{t-1}^T y_{t-1}},
\label{eq:gammat}
\end{align}
\end{subequations}
and
\begin{subequations} \label{eq:updates}
\begin{align}
S_t &\coloneqq \left[\bs_{t-\tilde{m}}, \bs_{t-\tilde{m}+1},\dotsc, \bs_{t-1}\right],\\
Y_t &\coloneqq \left[\by_{t-\tilde{m}}, \by_{t-\tilde{m}+1},\dotsc, \by_{t-1}\right],\\
D_t &\coloneqq \diag\left(\bs_{t-\tilde{m}}^T \by_{t-\tilde{m}}, \dotsc,
	\bs_{t-1}^T \by_{t-1}\right),\\
\left(L_t\right)_{i,j} &\coloneqq
	\begin{cases}
	\bs_{t - m - 1 + i}^T \by_{t - m - 1 + j}, &\text{ if } i > j,\\
	0, &\text{ otherwise.}
	\end{cases}
\end{align}
\end{subequations}
At the first iteration where no $\bs_i$ and $\by_i$ are available, we
set $H_0 \coloneqq a_0 I$ for some positive scalar $a_0$.
When $f$ is twice-differentiable and convex, we use
\begin{equation}
a_0 \coloneqq \frac{\left\|\nabla f(\bw_0)\right\|^2_{\nabla^2 f(\bw_0)}}{\left\|\nabla
f(\bw_0)\right\|^2}.
\label{eq:a0}
\end{equation}

If $f$ is not strongly convex, it is possible that \eqref{eq:Hk} is
only positive semi-definite.  In this case, we follow \citet{LiF01a},
taking the $m$ update pairs to be the most recent $m$ iterations for
which the inequality
\begin{equation}
\bs_i^T \by_i \geq \delta \bs_i^T \bs_i
\label{eq:safeguard}
\end{equation}
is satisfied, for some predefined $\delta > 0$. It can be shown that
this safeguard ensures that $H_t$ are always positive definite and the
eigenvalues are bounded within a positive range (see, for example, the
appendix of \citet{LeeW17a}).

No additional communication is required to compute $H_t$. The
gradients at all previous iterations have been shared with all machines
through the {\em allreduce} operation, and the iterates $\bw_t$ are
also available on each machine, as they are needed to compute the
local gradient. Thus the information needed to form $H_t$ is
available locally on each machine.

We now consider the costs associated with the matrix $M_t$. In
practice, $m$ is usually much smaller than $d$, so the
$O(m^3)$ cost of inverting the matrix directly is insignificant
compared to the cost of the other steps. However, if $d$ is large, the
computation of the inner products $\bs_i^T \by_j$ and $\bs_i^T \bs_j$
can be expensive. We can significantly reduce this cost by computing
and maintaining the inner products in parallel and assembling the
results with $O(m)$ communication cost. At the $t$th iteration,
given the new $\bs_{t-1}$, we compute its inner products with both
$S_t$ and $Y_t$ in parallel via the summations
\begin{equation*}
	\sum_{k=1}^K \left( (S_t)_{J_k,:}^T (\bs_{t-1})_{J_k} \right),\quad
	\sum_{k=1}^K \left( (Y_t)_{J_k,:}^T (\bs_{t-1})_{J_k} \right),
\end{equation*}
requiring $O(m)$ communication of the partial sums on each machine. We keep these results until $\bs_{t-1}$ and
$\by_{t-1}$ are discarded, so that at each iteration, only $2m$ (not
$O(m^2)$) inner products are computed.


\subsection{Solving the Subproblem}
\label{subsec:sparsa}
The approximate Hessian $H_t$ is generally not diagonal, so there is
no easy closed-form solution to \eqref{eq:quadratic}.  We will instead
use iterative algorithms to obtain an approximate solution to this
subproblem.
In single-core environments, coordinate descent (CD) is one of the
most efficient approaches for solving \eqref{eq:quadratic}
\citep{YuaHL12a, KaiYDR14a, SchT16a}. Since the subproblem
\eqref{eq:quadratic} is formed locally on all machines, a local CD
process can be applied when $g$ is separable and $d$ is not too large.
The alternative approach of applying proximal-gradient
methods  to \eqref{eq:quadratic} may be more efficient in
distributed settings, since they can be parallelized with
little communication cost for large $d$,
and can be applied to larger classes of regularizers $g$.


The fastest proximal-gradient-type methods are accelerated gradient
(AG) \citep{Nes13a} and SpaRSA \citep{WriNF09a}. SpaRSA is a basic
proximal-gradient method with spectral initialization of the parameter
in the prox term. SpaRSA has a few key advantages over AG despite
its weaker theoretical convergence rate guarantees. It tends to be
faster in the early iterations of the algorithm \citep{YanZ11a}, thus
possibly yielding a solution of acceptable accuracy in fewer
iterations than AG. It is also a descent method, reducing the
objective $Q_H$ at every iteration, which ensures that the solution
returned is at least as good as the original guess $\bp = 0$

In the rest of this subsection, we will describe a distributed
implementation of SpaRSA for \eqref{eq:quadratic}, with $H$ as
defined in \eqref{eq:Hk}. To distinguish between the iterations of our
main algorithm (i.e. the entire process required to update $\bw$ a
single time) and the iterations of SpaRSA, we will refer to them by
\emph{main iterations} and \emph{SpaRSA iterations} respectively.


Since $H$ and $\bw$ are fixed in this subsection, we will write
$Q_H(\cdot;\bw)$ simply as $Q(\cdot)$. We denote the $i$th iterate of
the SpaRSA algorithm as $\bp^{(i)}$, and we initialize $\bp^{(0)}
\equiv 0$. We denote the smooth part of $Q_H$ by $\hat f(\bp)$, and
the nonsmooth $g(\bw+\bp)$ by $\hat g(\bp)$. At the $i$th iteration of
SpaRSA, we define
\begin{equation}
\bu^{(i)}_{\psi_i} \coloneqq \bp^{(i)} -  \frac{\nabla \hat
	f(\bp^{(i)})}{ \psi_i},
	\label{eq:udef}
\end{equation}
and solve the following subproblem:
\begin{equation}
\bp^{(i+1)} = \arg \min_{\bp} \frac{1}{2} \left\|\bp -
\bu^{(i)}_{\psi_i}
	\right\|^2 + \frac{\hat g(\bp)}{\psi_i},
\label{eq:dk}
\end{equation}
where $\psi_i$ is defined by the following ``spectral'' formula:
\begin{equation}
	\psi_i = \frac{\left(\bp^{(i)} - \bp^{(i-1)}\right)^T \left(\nabla
	\hat f(\bp^{(i)}) - \nabla \hat
	f(\bp^{(i-1)})\right)}{\left\|\bp^{(i)} - \bp^{(i-1)}\right\|^2}.
	\label{eq:psi}
\end{equation}
When $i=0$, we use a pre-assigned value for $\psi_0$ instead.  (In our
LBFGS choice for $H_t$, we use the value of $\gamma_t$ from
\eqref{eq:gammat} as the initial estimate of $\psi_0$.)  The exact minimizer of
\eqref{eq:dk} can be difficult to compute for general regularizers $g$.
However, approximate solutions of \eqref{eq:dk} suffice to provide a convergence rate guarantee for solving \eqref{eq:quadratic} \citep{SchRB11a, SchT16a, GhaS16a, LeeW18a}. 
Since it is known (see \citep{LiF01a}) that the eigenvalues of $H$ are upper- and
lower-bounded in a positive range after the safeguard
\eqref{eq:safeguard} is applied, we can guarantee that this initialization of $\psi_i$
is bounded within a positive range; see Section~\ref{sec:analysis}.
The initial value of $\psi_i$ defined in \eqref{eq:psi} is increased
successively by a chosen constant factor $\beta>1$, and $\bp^{(i+1)}$
is recalculated from \eqref{eq:dk}, until the following sufficient
decrease criterion is satisfied:
\begin{equation}
	Q\left(\bp^{(i+1)}\right) \leq Q\left(\bp^{(i)}\right) -
	\frac{\psi_i \sigma_0}{2} \left\|\bp^{(i+1)} -
	\bp^{(i)}\right\|^2,
	\label{eq:accept}
\end{equation}
for some specified $\sigma_0 \in (0,1)$.
Note that the evaluation of $Q(\bp)$ needed in \eqref{eq:accept} can
be done efficiently through a parallel computation of $(\nabla \hat
f(\bp) + \nabla \tilde f(\bw))^T \bp / 2$ plus the $\hat g(\bp)$ term.
From the boundedness of $H$, one can
easily prove that \eqref{eq:accept} is satisfied after a finite number of increases of  $\psi_i$, as we will show in
Section~\ref{sec:analysis}.  In our algorithm, SpaRSA runs until
either a fixed number of iterations is reached, or when some certain
inner stopping condition for optimizing \eqref{eq:quadratic} is satisfied.

For general $H$, the computational bottleneck of $\nabla \hat f$ would
take $O(d^2)$ operations to compute the $H\bp^{(i)}$ term. However,
for our LBFGS choice of $H_k$, this cost is reduced to $O(md + m^2)$
by utilizing the matrix structure, as shown in the following formula:
\begin{align}
\nabla \hat f\left(\bp\right) = \nabla \tilde f\left(\bw\right)
+ H \bp
= \nabla \tilde f(\bw) + \gamma \bp - U_t \left(M_t^{-1} \left(U_t^T
\bp\right)\right).
\label{eq:u}
\end{align}
The computation of \eqref{eq:u} can be parallelized, by first
parallelizing computation of the inner product $U_t^T \bp^{(i)}$ via
the formula
\begin{equation*}
	\sum_{k=1}^K \left(U_t\right)_{J_k,:}^T \bp^{(i)}_{J_k}
\end{equation*}
with $O(m)$ communication.
(We implement the parallel inner products as described in
Section~\ref{subsec:products}.)  We either
construct the whole vector $\bu$ in \eqref{eq:udef} on all machines,
or let each machine compute a subvector of $\bu$ in \eqref{eq:udef}.
The former scheme is most suitable when $g$ is non-separable, but the
latter has a lower computational burden per machine, in cases for
which it is feasible to apply. We describe the latter scheme in more
detail.  The $k$th machine locally computes $\bp^{(i)}_{J_k}$ without
communicating the whole vector. Then at each iteration of SpaRSA,
partial inner products between $(U_t)_{J_k,:}$ and $\bp^{(i)}_{J_k}$
can be computed locally, and the results are assembled with one $O(m)$
communication. This technique also suggests a spatial advantage of our
method: The rows of $S_t$ and $Y_t$ can be stored in a distributed
manner consistent with the subvector partition.
This approach incurs $O(m)$ communication cost per SpaRSA iteration,
with the computational cost reduced from $O(md)$ to $O(md/K)$ per
machine.
Since both the $O(m)$ communication cost and the $O(md/K)$
computational cost are inexpensive when $m$ is small, in comparison to
the computation of $\nabla \tilde f$, one can afford to conduct
multiple iterations of SpaRSA at every main iteration.
Note that the latency incurred at every communication as
discussed in \eqref{eq:commcost} can be capped by setting a
maximum iteration limit for SpaRSA.
Finally, after the SpaRSA procedure terminates, all machines conduct
one $O(d)$ communication to gather the  update step $\bp$.

The distributed implementation of SpaRSA for solving \eqref{eq:quadratic}
is summarized in Algorithm~\ref{alg:sparsa}.

\begin{algorithm}[tb]
	\DontPrintSemicolon
\caption{Distributed SpaRSA for solving \eqref{eq:quadratic} with
LBFGS quadratic approximation on machine $k$}
\label{alg:sparsa}
\begin{algorithmic}[1]
\STATE Given $\beta > 1$, $\sigma_0 \in (0,1)$,
$M_t^{-1}$, $U_t$, and $\gamma_t$;
\STATE Set $\bp^{(0)}_{J_k} \leftarrow 0$;
	\FOR{$i=0,1,2,\dotsc$}
		\IF{$i=0$}
			\STATE $\psi = \gamma_t$;
		\ELSE
			\STATE Compute $\psi$ in \eqref{eq:psi} through
			\begin{equation*}
				\sum_{j=1}^K \left(\bp^{(i)}_{J_j} -
				\bp^{(i-1)}_{J_j}\right)^T
				\left(\nabla_{J_j} \hat f \left(\bp^{(i)} \right) - \nabla_{J_j}
				\hat f \left(\bp^{(i-1)} \right)\right),
			\end{equation*}
			and
			\begin{equation*}
				\sum_{j=1}^K
				\left\|\bp^{(i)}_{J_j} - \bp^{(i-1)}_{J_j}\right\|^2;
			\end{equation*}
			\Comment*[r]{
			$O(1)$
		comm.}
		\ENDIF
		\STATE Obtain \Comment*[r]{$O(m)$ comm.}
		\begin{equation*}
			U_t^T \bp^{(i)} = \sum_{j=1}^K \left(U_t\right)_{J_j,:}^T
			\bp^{(i)}_{J_j};
		\end{equation*}
		\STATE Compute \[\nabla_{J_k} \hat
		f\left(\bp^{(i)}\right) = \nabla_{J_k} \tilde f \left(\bw\right) + \gamma
		\bp^{(i)}_{J_k} -
	\left(U_t\right)_{J_k,:}\left(M_t^{-1} \left(U_t^T
\bp^{(i)}\right)\right)\]
by \eqref{eq:u};
		\WHILE{TRUE}
			\STATE Solve \eqref{eq:dk} on coordinates indexed by $J_k$ to
			obtain $\bp_{J_k}$;
			\IF{\eqref{eq:accept} holds \Comment*[r]{$O(1)$ comm.}}
			\STATE $\bp^{(i+1)}_{J_k} \leftarrow \bp_{J_k}$; $\psi_i \leftarrow \psi$;
			\STATE Break;
			\ENDIF
			\STATE $\psi \leftarrow \beta \psi$;
			\STATE Re-solve \eqref{eq:dk} with the new $\psi$ to
			obtain a new $\bp_{J_k}$;
		\ENDWHILE
		\STATE Break if some stopping condition is met;
	\ENDFOR
	\STATE Gather the final solution $\bp$ \Comment*[r]{$O(d)$ comm.}
\end{algorithmic}
\end{algorithm}

\subsection{Line Search}
After obtaining an update direction $\bp^k$ by approximately
minimizing $Q_{H_k}(\cdot;\bw_k)$, a line search procedure is usually
needed to find a step size $\alpha_k$ that ensures sufficient decrease
in the objective value.
We follow \citet{TseY09a} by using a modified-Armijo-type backtracking
line search to find a suitable step size $\alpha$.  Given the current
iterate $\bw$, the update direction $\bp$, and parameters $\sigma_1,
\theta \in (0,1)$, we set
\begin{equation}
\Delta \coloneqq \nabla \tilde f\left( \bw \right)^T \bp + g\left(\bw +
\bp \right) - g\left( \bw \right)
\label{eq:Delta}
\end{equation}
and pick the step size as the largest of $\theta^0, \theta^1,\dotsc$
satisfying
\begin{equation}
F\left(\bw + \alpha \bp \right) \leq F\left( \bw \right) + \alpha
\sigma_1
\Delta.
\label{eq:armijo}
\end{equation}
The computation of $\Delta$ can again be done in a distributed manner.
First,  $X_k^T \bp$ can be computed locally on each machine, then the
first term in \eqref{eq:Delta} is obtained by sending a scalar over
the network.  When $g$ is block-separable, its computation can also be
distributed across machines.
The vector $X_k^T \bp$ is then used to compute the left-hand side of
\eqref{eq:armijo} for arbitrary values of $\alpha$.  Writing $X_k^T
(\bw + \alpha \bp) = \left( X_k^T \bw \right) + \alpha \left(X_k^T
\bp\right)$, we see that once $ X_k^T \bw$ and $X_k^T \bp$ are known,
we can evaluate $X_k^T (\bw + \alpha \bp)$ via an ``axpy'' operation
(weighted sum of two vectors).  Because $H_t$ defined in
\eqref{eq:Hk} attempts to approximate the real Hessian, the unit step
$\alpha=1$ frequently satisfies \eqref{eq:armijo}, so we use the value
$1$ as the initial guess.
Aside from the communication needed to compute the summation of the
$f_k$ terms in the evaluation of $F$, the only other communication
needed is to share the update direction $\bp$ if \eqref{eq:quadratic}
was solved in a distributed manner. Thus, two rounds of $O(d)$
communication are incurred per main iteration. Otherwise, if each
machine solves the same subproblem \eqref{eq:quadratic} locally, then
only one round of $O(d)$ communication is required.

Our distributed algorithm for \eqref{eq:f} is summarized in
Algorithm~\ref{alg:proximalbfgs}.

\begin{algorithm}[tb]
	\DontPrintSemicolon
\caption{A distributed proximal variable-metric LBFGS method with line
	search for \eqref{eq:f}}
\label{alg:proximalbfgs}
\begin{algorithmic}[1]
\STATE Given $\theta, \sigma_1 \in (0,1)$, $\delta > 0$, an initial
point $\bw=\bw_0$, distributed $X = [X_1,\dotsc,X_K]$;
\FOR{Machines $k=1,\dotsc,K$ in parallel}
	\STATE Compute $X_k^T \bw$ and $f_k(X_k^T \bw)$;
	\STATE $H \leftarrow a I$ for some $a > 0$ (use \eqref{eq:a0} if
possible);
\STATE Obtain $F(\bw)$; \Comment*[r]{$O(1)$ comm.}
	\FOR{$t=0,1,2,\dotsc$}
		\STATE Compute $\nabla \tilde f(\bw)$ through
		\eqref{eq:grad};
		\Comment*[r]{$O(d)$ comm.}
		\IF{$t \neq 0$ and \eqref{eq:safeguard} holds for $(\bs_{t-1},
		\by_{t-1})$}
			\STATE Update $U$, $M$, and $\gamma$ by
			\eqref{eq:M}-\eqref{eq:updates};
			\Comment*[r]{$O(m)$ comm.}
			\STATE Construct a new $H$ from \eqref{eq:Hk};
		\ENDIF
		\IF{$H = aI$}
			\STATE Solve \eqref{eq:quadratic} directly to obtain
			$\bp$;
		\ELSE
			\STATE Solve \eqref{eq:quadratic} using Algorithm~\ref{alg:sparsa} either in a distributed manner or locally
			to obtain $\bp$;
		\ENDIF
		\STATE Compute $X_k^T \bp$;
		\STATE Compute $\Delta$ defined in
		\eqref{eq:Delta}; \Comment*[r]{
		$O(1)$ comm.}
		\FOR{$i=0,1,\dotsc$}
			\STATE $\alpha = \theta^i$;
			\STATE Compute $(X_k^T \bw) + \alpha (X_k^T \bp)$;
			\STATE Compute $F(\bw + \alpha \bp)$; \Comment*[r]{$O(1)$
			comm.}
			\IF{$F(\bw + \alpha \bp) \leq F(\bw) + \sigma_1 \alpha \Delta$}
				\STATE $\bw \leftarrow \bw + \alpha \bp$, $F(\bw)
				\leftarrow F(\bw + \alpha \bp)$;
				\STATE $X_k^T \bw \leftarrow X_k^T \bw + \alpha X_k^T
				\bp$;
                                \STATE $\bw_{t+1} \leftarrow \bw$;
                                \STATE
                                $\bs_t \leftarrow \bw_{t+1}-\bw_t$,
                                $\by_t \leftarrow \nabla \tilde{f}(\bw_{t+1}) -
                                \nabla \tilde{f} (\bw_t)$;
				\STATE Break;
			\ENDIF
		\ENDFOR
	\ENDFOR
\ENDFOR
\end{algorithmic}
\end{algorithm}

\subsection{Cost Analysis}
We now summarize the costs of our algorithm.
For the distributed version of Algorithm
\ref{alg:sparsa},
each iteration costs
\begin{equation}
	O\left(\frac{d}{K} + \frac{md}{K} +  m^2\right) =
	O\left(\frac{md}{K} + m^2\right)
	\label{eq:costsparsa}
\end{equation}
in computation  and
\begin{equation*}
	O\left(m + 1 \times \text{number of times \eqref{eq:accept} is
evaluated}\right)
\end{equation*}
in communication.
In the next section, we will show that \eqref{eq:accept} is accepted
in a constant number of times and thus the overall communication cost
is $O(m)$.

For  Algorithm \ref{alg:proximalbfgs}, the computational cost per
iteration is
\begin{equation}
	O\left(\frac{\#\text{nnz}}{K} + \frac{n}{K} + d + \frac{md}{K} +
	\frac{d}{K}\right) = O\left( \frac{\#\text{nnz}}{K} + d +
	\frac{md}{K} \right),
	\label{eq:costmain}
\end{equation}
where \#nnz is the number of nonzero elements in $X$,
and the communication cost is
\begin{equation*}
	O\left( 1 + m + d \right) = O\left( d \right).
\end{equation*}
We note that the costs of Algorithm \ref{alg:sparsa} are
dominated by those of Algorithm \ref{alg:proximalbfgs} if a fixed number of SpaRSA
iterations is conducted every main iteration.
\section{Communication Complexity}
\label{sec:analysis}
The use of an iterative solver for the subproblem \eqref{eq:quadratic}
generally results in an inexact solution.  We first show that running
SpaRSA for any fixed number of iterations guarantees a step $\bp$
whose accuracy is sufficient to prove overall convergence.

\begin{lemma}
	\label{lemma:sparsa}
Using $H_t$ as defined in \eqref{eq:Hk} with the safeguard
mechanism \eqref{eq:safeguard} in \eqref{eq:quadratic}, we have the following.
\begin{compactenum}
\item
There exist constants $c_1 \geq c_2 > 0$ such that
$c_1 I \succeq H_t \succeq c_2 I$  for all main iterations.
Moreover, $\|X^T X\| L \geq \gamma_t \geq \delta$ for all
$t > 0$.
\item
The initial estimate of $\psi_i$ at every SpaRSA iteration
is bounded within the range of $[\min\{c_2, \delta\}, \max\{c_1, \|X^T
X\|L\}]$, and the final accepted value $\psi_i$ is upper-bounded.
\item
SpaRSA is globally Q-linear convergent in solving
\eqref{eq:quadratic}.
Therefore, there exists $\eta \in [0,1)$ such that if we run
	at least $S$ iterations of SpaRSA for all main iterations for any
	$S>0$, the approximate solution $\bp$ satisfies
\begin{equation}
-\eta^S Q^* = \eta^S\left(Q\left(0\right) - Q^*\right) \geq
Q\left(\bp\right) - Q^*,
\label{eq:approx}
\end{equation}
where $Q^*$ is the optimal objective of
\eqref{eq:quadratic}.
\end{compactenum}
\end{lemma}

Lemma \ref{lemma:sparsa} establishes how the number of iterations of
SpaRSA affects the inexactness of the subproblem solution.
Given this measure,
we can leverage the results developed in \citet{LeeW18a} to obtain
iteration complexity guarantees for our algorithm. Since in our
algorithm, communication complexity scales linearly with iteration
complexity, this guarantee provides a bound on the amount of
communication.
In particular, our method communicates
$O(d+mS)$ bytes per iteration (where $S$ is the number of SpaRSA
iterations used, as in Lemma~\ref{lemma:sparsa}) and the second term can usually
be ignored for small $m$.

We show next that the step size generated by our line search procedure
in Algorithm~\ref{alg:proximalbfgs} is lower bounded by a positive
value.
\begin{lemma}
\label{lemma:delta}
If SpaRSA is run at least $S$ iterations in solving
\eqref{eq:quadratic},
the corresponding $\Delta$ defined in \eqref{eq:Delta} satisfies
\begin{equation}
	\Delta \leq -\frac{c_2 \left\|d\right\|^2}{
		\left(1 + \eta^{{S}/{2}}\right)}.
\label{eq:deltabound}
\end{equation}
Moreover, the backtracking subroutine in Algorithm~\ref{alg:proximalbfgs} terminates in finite steps and produces a step
size
\begin{equation}
\alpha \geq \bar\alpha \geq \min\left\{1, \frac{2\theta\left(1 -
\sigma_1\right) c_2}{\left\|X^T X\right\|L\left(1 + \eta^{{S}/{2}}\right)}\right\}.
\label{eq:linesearchbound2}
\end{equation}
\end{lemma}
This result is just a worst-case guarantee; in practice we often
observe that the line search procedure terminates with $\alpha
=1$ for our choice of $H$, as we see  in our experiments.

Next, we analyze communication complexity of
Algorithm~\ref{alg:proximalbfgs}.
\begin{theorem}
If we apply Algorithm~\ref{alg:proximalbfgs} to solve \eqref{eq:f},
and Algorithm~\ref{alg:sparsa} is run for $S$ iterations at each main
iteration, then the following claims hold.
\begin{itemize}[
topsep=0pt,
partopsep=0pt,
leftmargin=*,
itemsep=0pt,
parsep=0pt
]
\item
  Suppose that the following variant of strong convexity holds: There
  exists $\mu > 0$ such that for any $\bw$ and any $a \in [0,1]$, we have
\begin{align}
\label{eq:strong}
&~F\left(a \bw + \left(1 - a\right)
P_{\Omega}\left(\bw\right)\right)\\
\leq &~a F\left(\bw\right) +
\left(1 - a\right) F^* - \frac{\mu a\left( 1 - a \right)}{2}
\left\|\bw - P_{\Omega}\left(\bw\right)\right\|^2,
\nonumber
\end{align}
where $F^*$ is the optimal objective value of \eqref{eq:f}, $\Omega$
is the solution set, and $P_{\Omega}$ is the projection onto this set.
Then Algorithm~\ref{alg:proximalbfgs} converges globally at a Q-linear
rate. That is, 
\begin{align}
\nonumber
\frac{F\left(\bw_{t+1}\right) - F^*}{F\left( \bw_t \right) - F^*}
&\leq
1 - \frac{\bar \alpha \sigma_1 \left(1 -\eta^S \right)\mu}{\mu +
c_1},\quad \forall t.
\label{eq:qlinear}
\end{align}
Therefore, to get an approximate solution of \eqref{eq:f} that
is $\epsilon$-accurate in the sense of objective value, we need to
perform at most
\begin{equation}
	O\left(\frac{\mu + c_1}{\mu \sigma_1 \bar{\alpha} \left( 1 -
\eta^S \right)} \log\frac{1}{\epsilon} \right)
\label{eq:commlinear}
\end{equation}
rounds of $O(d)$ communication.
\item When $F$ is convex, and the level set defined by $\bw_0$ is
  bounded, define
\begin{equation*}
R_0 \coloneqq \sup_{\bw: F\left( \bw \right) \leq F\left( \bw_0
	\right)}\quad\left\|\bw - P_\Omega(\bw)\right\|.
\end{equation*}
Then we obtain the following expressions for rate of convergence of
the objective value.
\begin{compactenum}
\item When $F(\bw_t) - F^* \geq c_1 R_0^2$,
\begin{equation*}
\frac{F(\bw_{t+1}) - F^*}{F(\bw_t) - F^*} \leq 1 - \frac{\left(
	1 - \eta^S \right)\sigma_1 \bar{\alpha}}{2}.
\end{equation*}
\item Otherwise, we have globally for all $t$ that
\begin{align*}
	\frac{F\left(\bw_t\right) - F^*}{2} &\leq \frac{c_1 R_0^2 +
	F\left(\bw_0\right) - F^*}{ \sigma_1 t(1 -
	\eta^S)\bar \alpha}.
\end{align*}
\end{compactenum}
This implies a communication complexity of
\begin{equation*}
	\begin{cases}
	O\left( \frac{2}{\left( 1 - \eta^S\right) \sigma_1 \bar{\alpha}
}\log \frac{1}{\epsilon} \right) &\text{ if } \epsilon \geq
	c_1 R_0^2,\\
	\frac{2 \left( c_1 R_0^2 + F(\bw_0) - F^* \right)}{\sigma_1 (1 -
	\eta^S) \bar \alpha \epsilon}&\text{ else}.
\end{cases}
\end{equation*}
\item If $F$ is non-convex, the norm of the proximal gradient steps
\begin{equation*}
G_t \coloneqq \arg\min_\bp \nabla f\left(\bw_t\right)^T \bp +
\frac{\|\bp\|^2}{2}
+ g\left(\bw_t + \bp\right)
\end{equation*}
converge to zero at a rate of $O(1 / \sqrt{t})$ in the following sense:
\begin{align*}
\min_{0 \leq i \leq t }\left\|G_i\right\|^2 \leq \frac{F\left( \bw_0
	\right) - F^*}{\gamma \left( t+1 \right)} \frac{c_1^2\left(1 +
	\frac{1}{c_2} + \sqrt{1 - \frac{2}{c_1} +
	\frac{1}{c_2^2}}\right)^2}{2 c_2 \bar \alpha (1 -
	\eta^S)}.
\end{align*}
\end{itemize}
\end{theorem}
Note that it is known that the norm of $G_t$ is zero if and only if
$\bw_t$ is a stationary point \citep{LeeW18a}, so this measure serves
as a first-order optimality condition.

 Our computational experiments cover the case of $F$ convex; 
exploration of the method on nonconvex $F$ is left for future work.

\section{Related Works}
\label{sec:related}
The framework of using \eqref{eq:quadratic} to generate update
directions for optimizing \eqref{eq:f} has been discussed in existing
works with different choices of $H$, but always in the single-core
setting. \citet{LeeSS14a} focused on using $\nabla^2 \tilde f$ as $H$,
and proved local convergence results under certain additional
assumptions.  In their experiment, they used AG to solve
\eqref{eq:quadratic}.  However, in distributed environments, using
$\nabla^2 \tilde f$ as $H$ incurs an $O(d)$ communication cost per AG
iteration in solving \eqref{eq:quadratic}, because computation of the
term $\nabla^2 \tilde f(\bw) \bp = X \nabla^2 f(X^T \bw) X^T \bp$
requires one {\em allreduce} operation to calculate a weighted sum of
the columns of $X$.

\citet{SchT16a} and \citet{GhaS16a} showed global convergence rate
results for a method based on \eqref{eq:quadratic} with bounded $H$,
and suggested using randomized coordinate descent to solve
\eqref{eq:quadratic}.  In the experiments of these two works, they
used the same choice of $H$ as we do in this paper, with CD as the
solver for \eqref{eq:quadratic}, which is well suited to their
single-machine setting.  Aside from our extension to the distributed
setting and the use of SpaRSA, the third major difference between their
algorithm and ours is that they do not conduct line search on the step
size.  Instead, when the obtained solution with a unit step size does
not result in sufficient objective value decrease, they add a scaled
identity matrix to $H$ and solve \eqref{eq:quadratic} again starting
from $\bp^{(0)} = 0$. The cost of repeatedly solving \eqref{eq:quadratic}
from scratch can be high, which results in an algorithm with higher
overall complexity. This potential inefficiency is  exacerbated further by the inefficiency
of coordinate descent in the distributed setting.

Our method can be considered as a special case of the algorithmic
framework in \citet{LeeW18a, BonLPP16a}, which both focus on
analyzing the theoretical guarantees under various conditions. In the
experiments of \citet{BonLPP16a}, $H$ is obtained from the diagonal entries of
$\nabla^2 \tilde f$, making the subproblem
\eqref{eq:quadratic} easy to solve, but this simplification does not
take full advantage of curvature information.
Although our theoretical convergence analysis follows directly from
\citet{LeeW18a}, that paper does not provide details of experimental results
or implementation, and its analysis focuses on general $H$
rather than the LBFGS choice we use here.

Some methods consider solving \eqref{eq:f} in a distributed
environment where $X$ is partitioned feature-wise (i.e. along rows
instead of columns). There are two potential disadvantages of this
approach.  First, new data points can easily be assigned to one of the
machines in our approach, whereas in the feature-wise approach, the
features of all new points would need to be distributed around the
machines.  Second, local curvature information is obtained, so the
update direction can be poor if the data is distributed nonuniformly
across features. (Data is more likely to be distributed evenly across
instances than across features.)  In the extreme case in which each
machine contains only one row of $X$, only the diagonal entries of the
Hessian can be obtained locally, so the method reduces to a scaled
version of proximal gradient.

\section{Numerical Experiments}
\label{sec:exp}
We investigate the empirical performance of
Algorithm~\ref{alg:proximalbfgs} in solving $\ell_1$-regularized
logistic regression problems. The code used in our experiment is
available at \url{http://github.com/leepei/dplbfgs/}.  Given training
data points $(\bx_i, y_i) \in \R^d \times \{-1,1\}$ for
$i=1,\dotsc,n$, the objective function is
\begin{equation}
F(\bw) = C\sum_{i=1}^n \log\left(1 + e^{-y_i \bx_i^T
	\bw}\right) + \|\bw\|_1,
\label{eq:logistic}
\end{equation}
where $C > 0$ is a parameter prespecified to trade-off between
the loss term and the regularization term. We fix $C$ to $1$
for simplicity in our experiments.
We consider the publicly available binary classification data sets
listed in Table~\ref{tbl:data}\footnote{Downloaded from
  \url{https://www.csie.ntu.edu.tw/~cjlin/libsvmtools/datasets/}.},
and partitioned the instances evenly across machines.

The parameters of our algorithm were set as follows:
$\theta = 0.5$, $\beta = 2$, $\sigma_0 = 10^{-2}$, $\sigma_1 = 10^{-4}$,
$m=10$, $\delta = 10^{-10}$.
The parameters in SpaRSA follow the setting in \cite{WriNF09a},
$\theta$ is set to halve the step size each time, the value of
$\sigma_0$ follows the default experimental setting of
\cite{LeeWCL17a}, $\delta$ is set to a small enough number, and $m=10$
is a common choice for LBFGS.

\begin{table}[tb]
\caption{Data statistics.}
\label{tbl:data}
\centering
\begin{tabular}{@{}l|rr|rr}
Data set & $n$ & $d$ & \#nonzeros\\
\hline
news & 19,996 & 1,355,191 & 9,097,916 \\
epsilon & 400,000 & 2,000 & 800,000,000\\
webspam & 350,000 & 16,609,143 & 1,304,697,446\\
avazu-site & 25,832,830 & 999,962 & 387,492,144
\end{tabular}
\end{table}

We ran our experiments on a local cluster of $16$ machines running
MPICH2, and all algorithms are implemented in C/C++.  The inversion of
$M$ defined in \eqref{eq:M} is performed through LAPACK
\citep{And99a}.  The comparison criteria are the relative objective
error $(F(\bw) - F^*) / F^*$,
versus either the amount communicated (divided by $d$) or the overall
running time.  The former criterion is useful in estimating the
performance in environments in which communication cost is extremely
high.

\subsection{Effect of Inexactness in the Subproblem Solution}
We first examine how the degree of inexactness of the approximate
solution of subproblems \eqref{eq:quadratic} affects the convergence
of the overall algorithm. Instead of treating SpaRSA as a steadily
linearly converging algorithm, we take it as an algorithm that
sometimes decreases the objective much faster than the worst-case
guarantee, thus an adaptive stopping condition is used.  In
particular, we terminate Algorithm \ref{alg:sparsa} when the norm of
the current update step is smaller than $\epsilon_1$ times that of the
first update step, for some prespecified $\epsilon_1 > 0$.  From the
proof of Lemma~\ref{lemma:sparsa}, the norm of the update step bounds
the value of $Q(\bp) - Q^*$ both from above and from below, and thus
serves as a good measure of the solution precision.  In
Table~\ref{tbl:stop}, we compare runs with the values $\epsilon_1 =
10^{-1}, 10^{-2}, 10^{-3}$. For the datasets news20 and webspam, it is
as expected that tighter solution of \eqref{eq:quadratic} results in
better updates and hence lower communication cost. This may not result
in a longer convergence time. As for the dataset epsilon, which has a
smaller data dimension $d$, the $O(m)$ communication cost per SpaRSA
iteration for calculating $\nabla \tilde f$ is significant in
comparison. In this case, setting a tighter stopping criteria for
SpaRSA can result in higher communication cost and longer running
time.

In Table~\ref{tbl:steps}, we show the distribution of the step sizes
over the main iterations, for the same set of values of $\epsilon_1$.  As
we discussed in Section~\ref{sec:analysis}, although the smallest
$\alpha$ can be much smaller than one, the unit step is usually
accepted.  Therefore, although the worst-case communication complexity
analysis is dominated by the smallest step encountered, the practical
behavior is much better.

\begin{table}[tb]
\caption{Different stopping conditions of SpaRSA as an approximate
solver for \eqref{eq:quadratic}. We show required amount of
communication (divided by $d$) and running time (in seconds) to reach
$F(\bw_t) - F^* \leq 10^{-3} F^*$.}
\label{tbl:stop}
\centering
	\begin{tabular}{l|r|r|r}
Data set & $\epsilon_1$ & Communication & Time\\
\hline
\multirow{3}{*}{news20} & $10^{-1}$ & 28 & 11 \\
& $10^{-2}$ & 25 & 11 \\
& $10^{-3}$ & 23 & 14 \\
\hline
\multirow{3}{*}{epsilon} & $10^{-1}$ & 144 & 45 \\
& $10^{-2}$ & 357 & 61 \\
& $10^{-3}$ & 687 & 60 \\
\hline
\multirow{3}{*}{webspam} & $10^{-1}$ & 452 & 3254 \\
& $10^{-2}$ & 273 & 1814 \\
& $10^{-3}$ & 249 & 1419
\end{tabular}
\end{table}
\begin{table}[tb]
\caption{Step size distributions.}
\label{tbl:steps}
\centering
\begin{tabular}{l|r|r|r}
Data set & $\epsilon_1$ & percent of $\alpha = 1$ & smallest $\alpha$\\
\hline
\multirow{3}{*}{news20} & $10^{-1}$ & $95.5\%$ & $2^{-3}$ \\
& $10^{-2}$ & $95.5\%$ & $2^{-4}$ \\
& $10^{-3}$ & $95.5\%$ & $2^{-3}$ \\
\hline
\multirow{3}{*}{epsilon} & $10^{-1}$ & $96.8\%$ & $2^{-5}$ \\
& $10^{-2}$ & $93.4\%$ & $2^{-6}$ \\
& $10^{-3}$ & $91.2\%$ & $2^{-3}$ \\
\hline
\multirow{3}{*}{webspam} & $10^{-1}$ & $98.5\%$ & $2^{-3}$ \\
& $10^{-2}$ & $97.6\%$ & $2^{-2}$ \\
& $10^{-3}$ & $97.2\%$ & $2^{-2}$
\end{tabular}
\end{table}

\subsection{Comparison with Other Methods}
Now we compare our method with two state-of-the-art distributed
algorithms for \eqref{eq:f}.  In addition to a proximal-gradient-type
method that can be used to solve general \eqref{eq:f} in distributed
environments easily, we also include one solver specifically designed
for $\ell_1$-regularized problems in our comparison.  These methods
are:
\begin{itemize}[
topsep=0pt,
partopsep=0pt,
leftmargin=*,
itemsep=0pt,
parsep=0pt
]
	\item DPLBFGS: our Distributed Proximal LBFGS approach. We fix
		$\epsilon_1 = 10^{-2}$ in this experiment.
	\item SPARSA \citep{WriNF09a}: the method described in Section~\ref{subsec:sparsa}, but applied directly to
		\eqref{eq:f}.
	\item OWLQN \citep{AndG07a}: an orthant-wise quasi-Newton method
		specifically designed for $\ell_1$-regularized problems.
		We fix $m=10$ in the LBFGS approximation.
\end{itemize}
We implement all methods in C/C++ and MPI. Note that the AG method
\citep{Nes13a} can also be used, but its empirical performance has been
shown to be similar to SpaRSA \citep{YanZ11a} and it requires strong
convexity and Lipschitz parameters to be estimated, which induces an
additional cost.
A further examination on different values of $m$ indicates that
  convergence speed of our method improves with larger $m$,
while in OWLQN, larger $m$  usually does not
lead to better results.
We use the same value of $m$ for both
methods and choose a value that favors OWLQN.

The results are provided in Figure~\ref{fig:compare}. 
Our method is always the fastest in both criteria.  For
epsilon, our method is orders of magnitude faster, showing that
correctly using the curvature information of the smooth part is indeed
beneficial in reducing the communication complexity.

It is possible to include specific heuristics for
$\ell_1$-regularized problems, such as those applied in
\citet{YuaHL12a, KaiYDR14a}, to further accelerate our method,
and the exploration on this direction is an interesting topic for
future work.

\begin{figure}[tb]
\centering
\begin{tabular}{cc}
	Communication & Time\\
	\multicolumn{2}{c}{news20}\\
	\includegraphics[width=.42\linewidth]{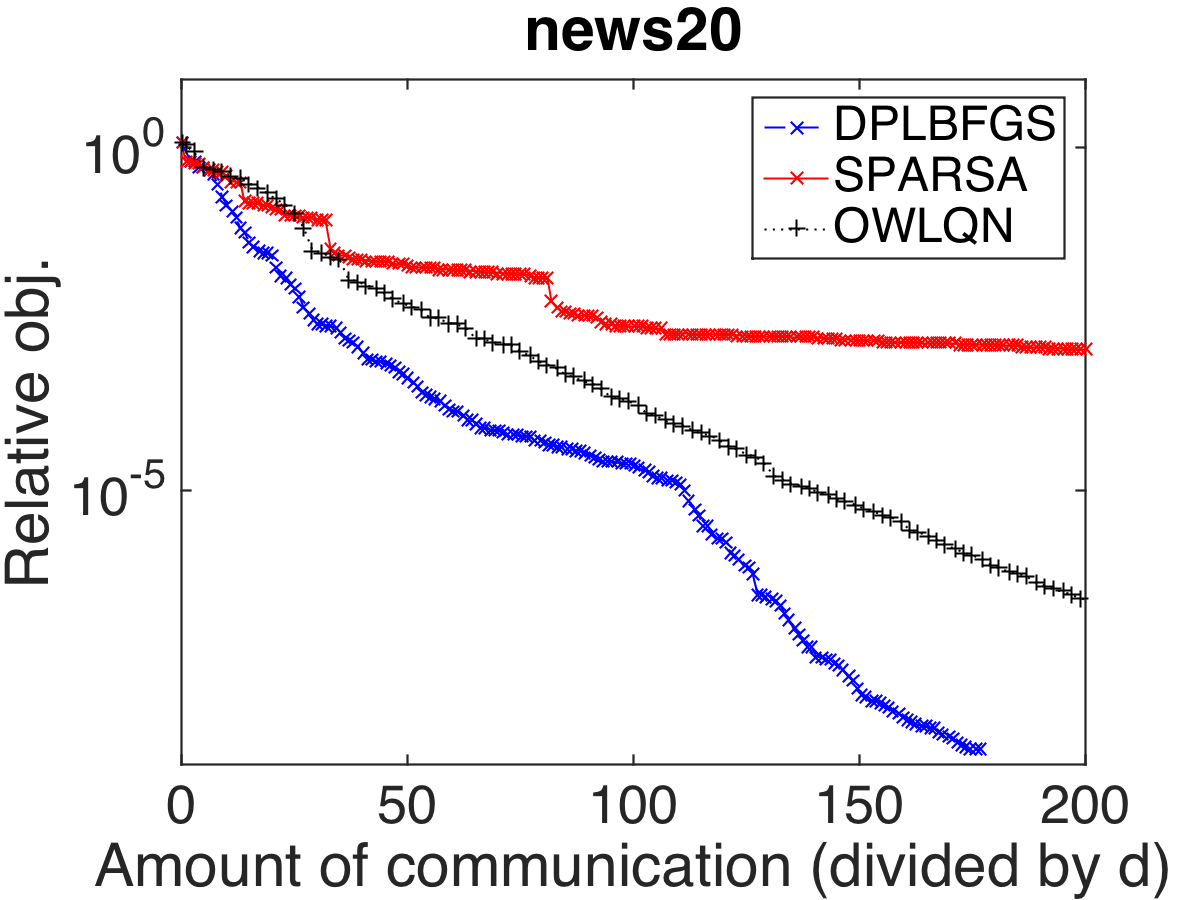}&
	\includegraphics[width=.42\linewidth]{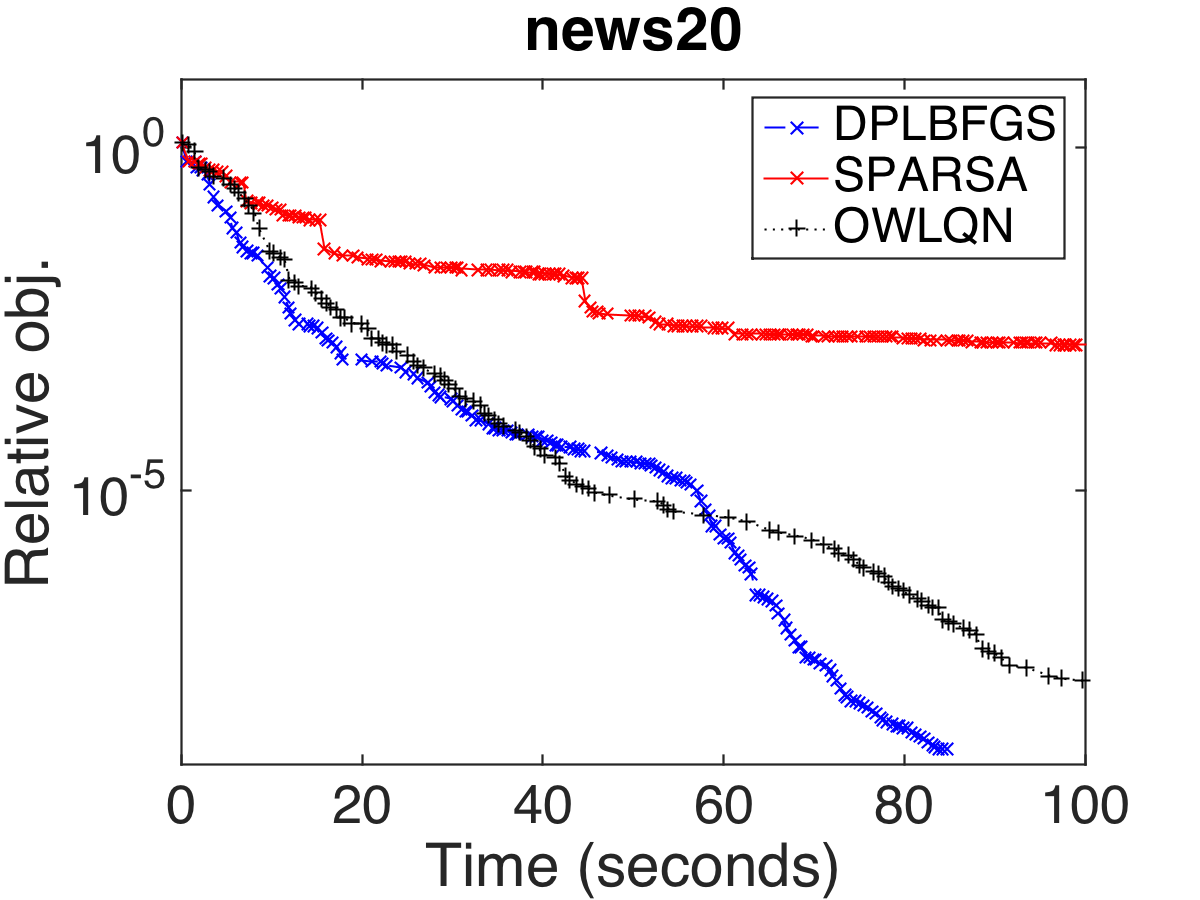}\\
	\multicolumn{2}{c}{epsilon}\\
	\includegraphics[width=.42\linewidth]{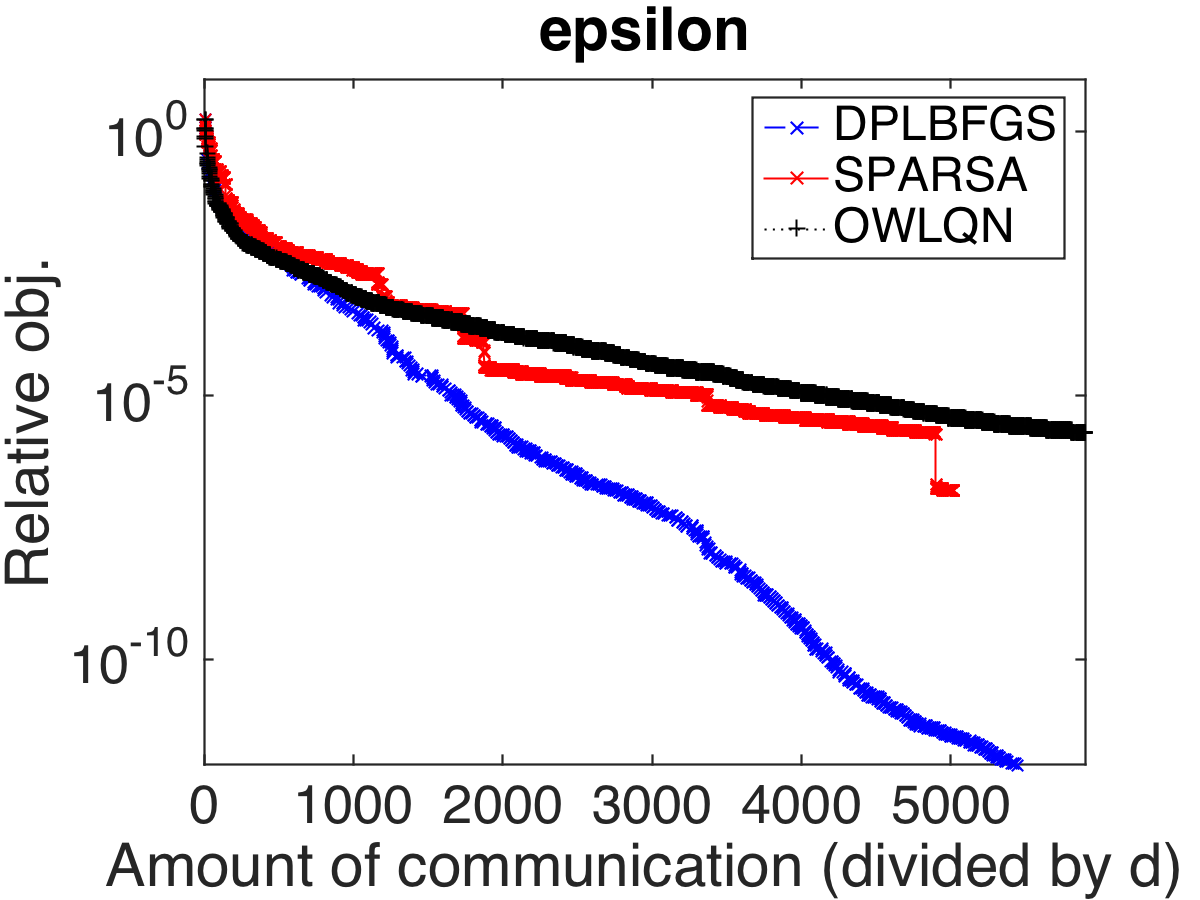}&
	\includegraphics[width=.42\linewidth]{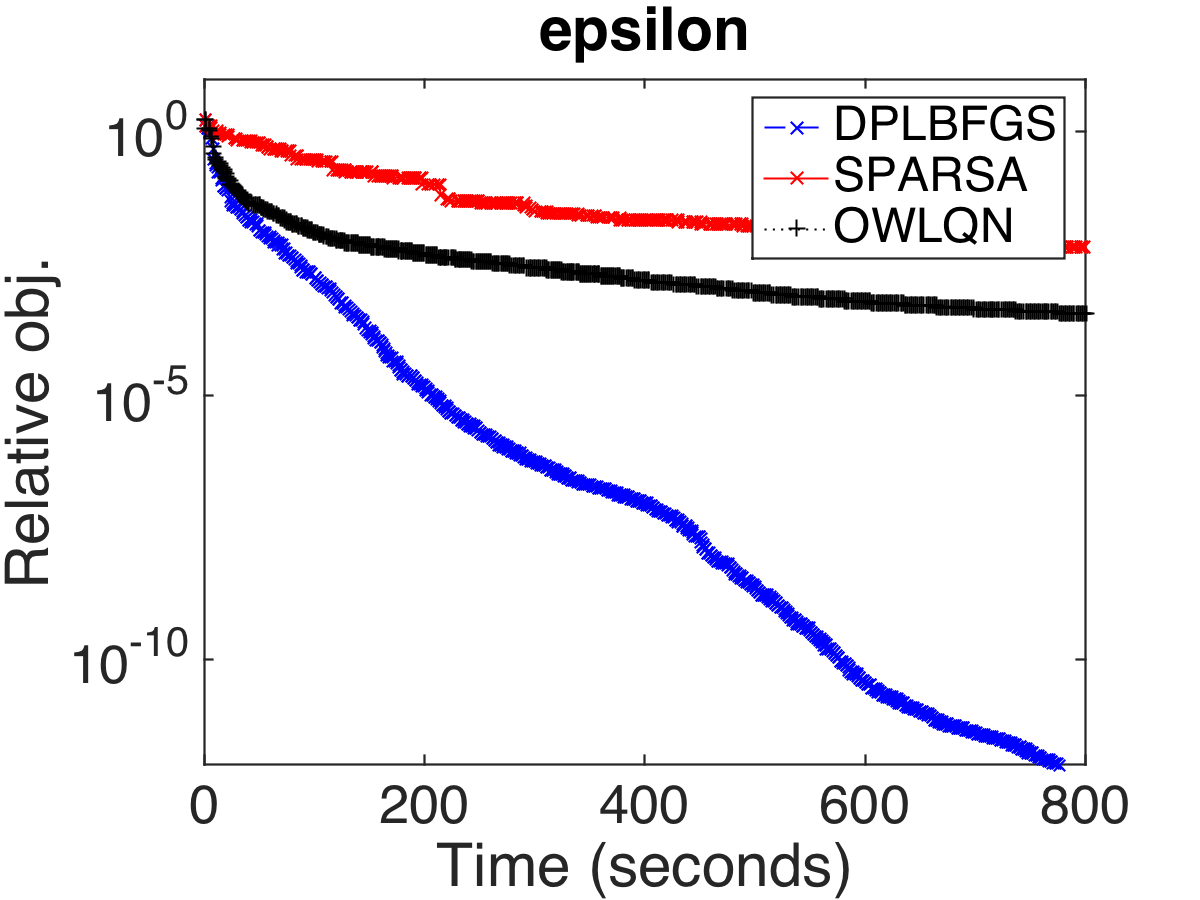}\\
	\multicolumn{2}{c}{webspam}\\
	\includegraphics[width=.42\linewidth]{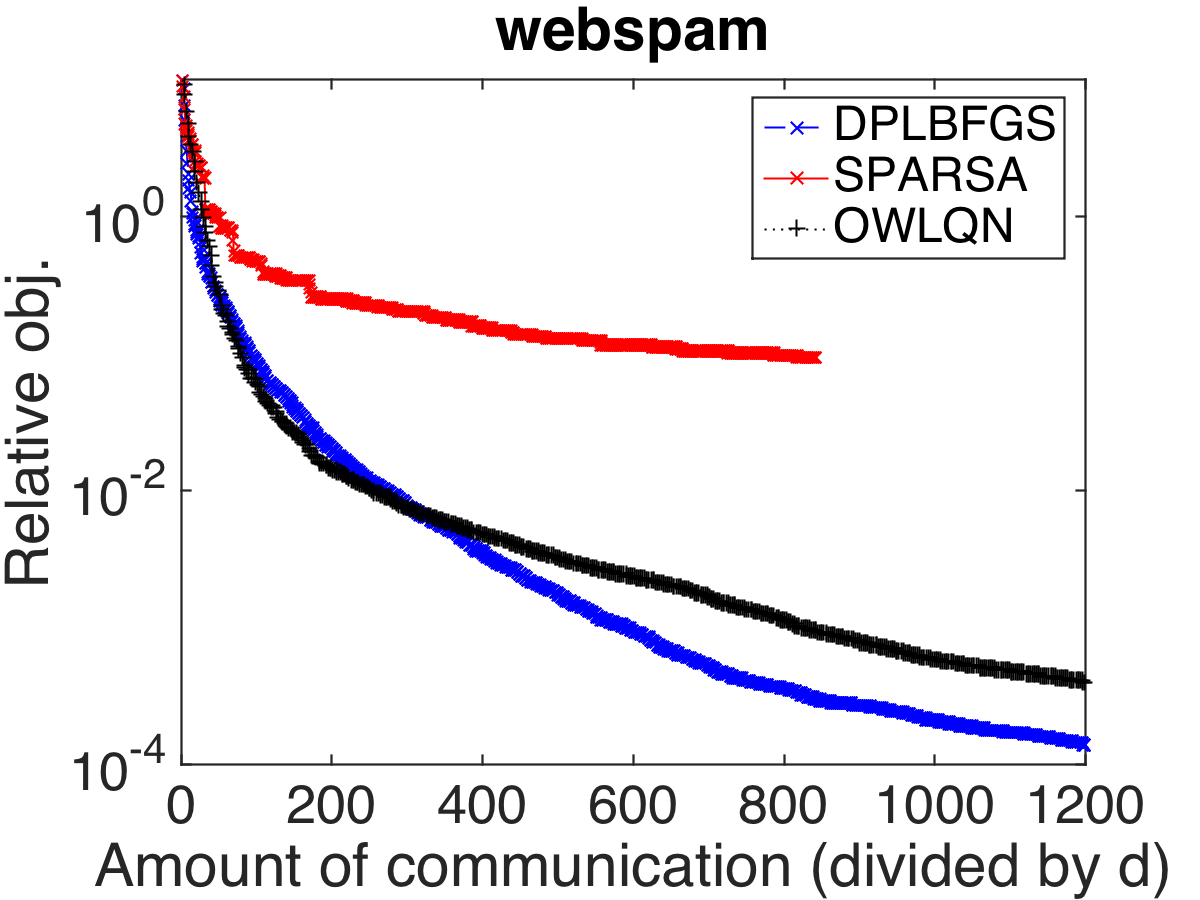}&
	\includegraphics[width=.42\linewidth]{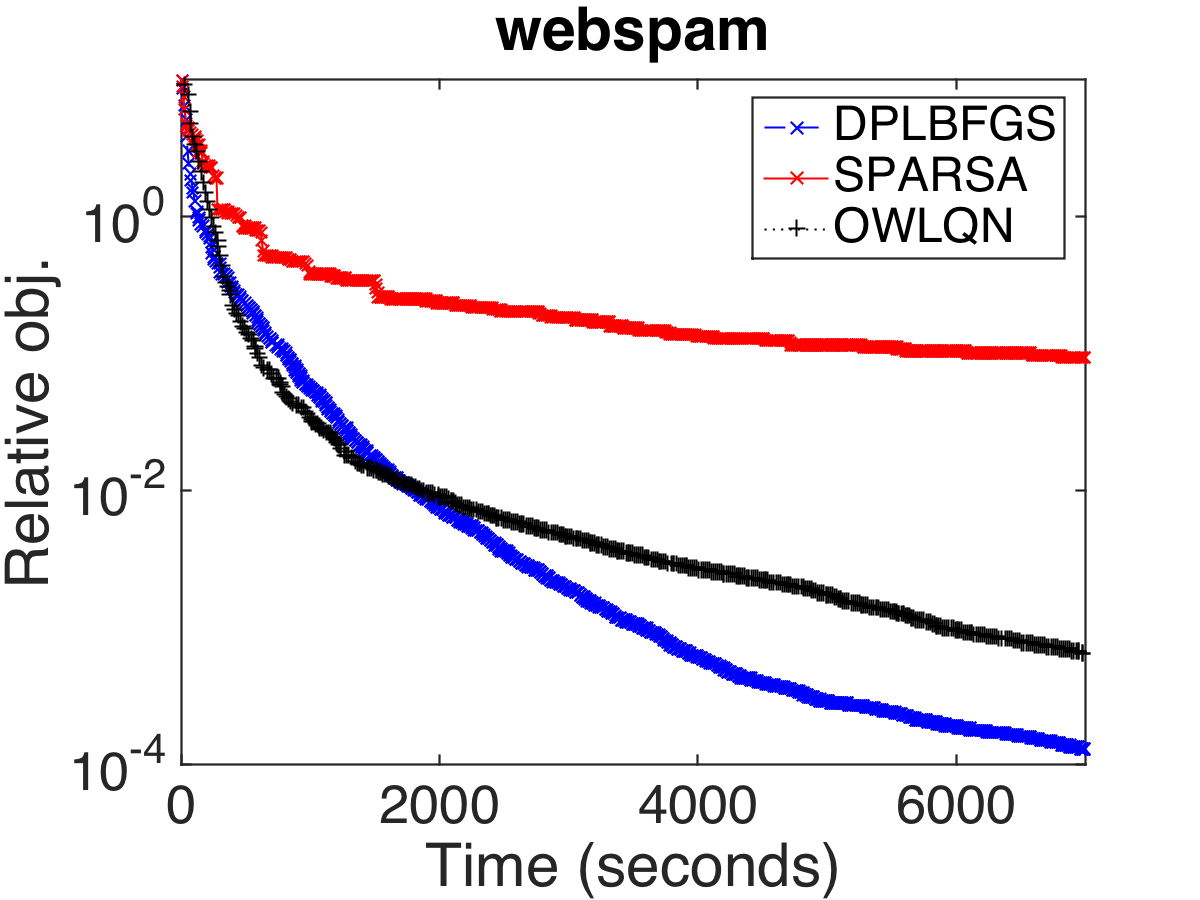}\\
	\multicolumn{2}{c}{avazu-site}\\
	\includegraphics[width=.42\linewidth]{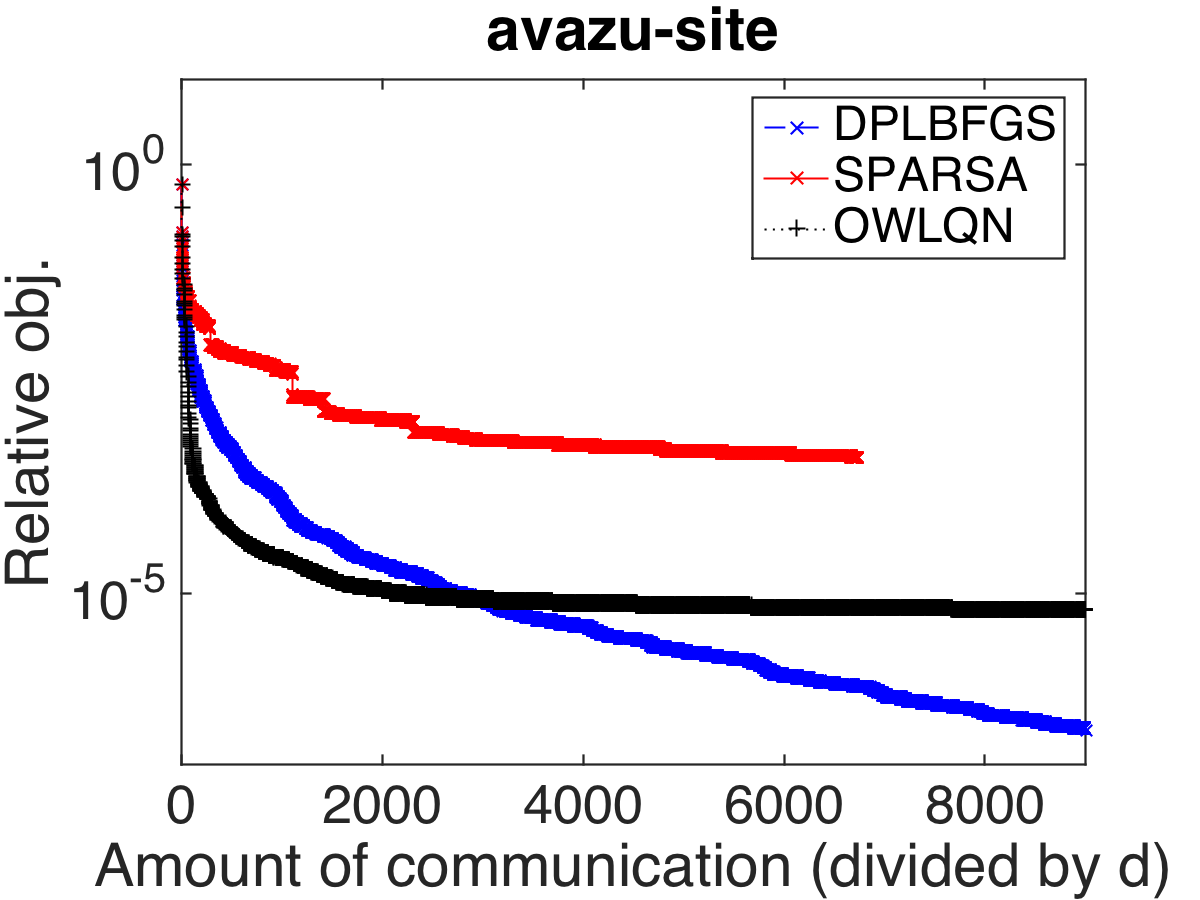}&
	\includegraphics[width=.42\linewidth]{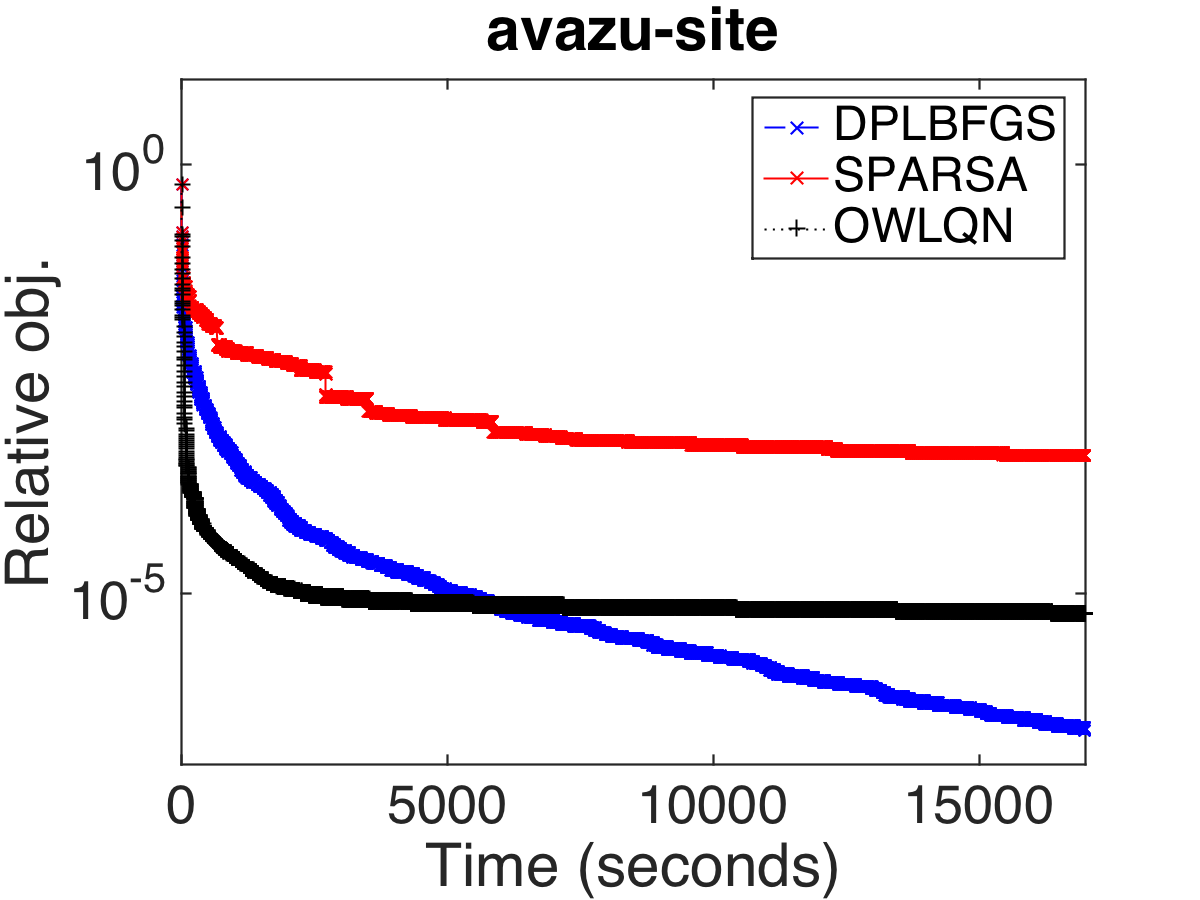}
\end{tabular}
	\caption{Comparison between different methods for $\ell_1
		$-regularized logistic regression in terms of relative
		objective difference to the optimum. Left: communication
	(divided by $d$); right: running time (in seconds).}
\label{fig:compare}
\end{figure}

\section{Conclusions} \label{sec:conclusions}
In this work, we propose a practical and communication-efficient
distributed algorithm for solving general regularized nonsmooth ERM
problems.  Our algorithm enjoys fast performance both theoretically
and empirically and can be applied to a wide range of ERM problems.
It is possible to extend our approach for solving the distributed dual
ERM problem with a strongly convex primal regularizer, and we expect
our framework to outperform state of the art, which only uses
block-diagonal parts of the Hessian that can be computed locally.
These topics are left for future work.

\section*{Acknowledgement}
This work was supported by NSF Awards IIS-1447449, 1628384, 1634597,
and 1740707; AFOSR Award FA9550-13-1-0138; and Subcontract 3F-30222
from Argonne National Laboratory.
The authors thank En-Hsu Yen for fruitful discussions.
\bibliographystyle{./ACM-Reference-Format}
\bibliography{inexactprox,distributederm}
\appendix
\section{Proofs}
We provide proof of Lemma~\ref{lemma:sparsa} in this section.
The rest of Section~\ref{sec:analysis} directly follows the results in
\citet{LeeW18a} by noting that $\nabla \tilde f(\bw)$ is $\|X^T X\|
L$-Lipschitz continuous, and are therefore omitted.
\begin{proof}[Proof of Lemma~\ref{lemma:sparsa}]
We prove the three results separately.
\begin{enumerate}
\item The boundedness of $H_t$ directly follow from the results in
	\citet{LiF01a}.  A more detailed proof can be found in, for
	example, \citet[Appendix E]{LeeW17a}.
	The lower bound of $\gamma_t$ is directly through
	\eqref{eq:safeguard}, and the upper bound is from the Lipschitz
	continuity of $\nabla \tilde f$.
\item By directly expanding $\nabla \hat f$, we have that for
	any $\bp_1, \bp_2$,
	\begin{align*}
	\nabla \hat f(\bp_1) - \nabla \hat f(\bp_2)
	&= \nabla \tilde f(\bw) + H \bp_1 - \left(\nabla \tilde f(\bw) +
	H \bp_2\right)\\
	&= H (\bp_1 - \bp_2).
	\end{align*}
	Therefore, we have
	\begin{equation*}
		\frac{\left(\nabla \hat f(\bp_1) - \nabla \hat
			f(\bp_2)\right)^T \left( \bp_1 - \bp_2 \right)}{\left\|\bp_1 -
			\bp_2\right\|^2}
			= \frac{\left\|\bp_1 - \bp_2\right\|_H^2}{\left\|\bp_1 -
			\bp_2\right\|^2}
			\in \left[c_2, c_1\right]
	\end{equation*}
	for bounding $\psi_i$ for $i > 0$, and the bound for $\psi_0$ is
	directly from the bounds of $\gamma_t$.
	The combined bound is therefore $[\min\{c_2, \delta\}, \max \{c_1,
	\|X^T X\| L\}]$.
	Next, we show that the final $\psi_i$ is always upper-bounded.
	The right-hand side of \eqref{eq:dk} is equivalent to the
	following:
	\begin{equation}
	\arg \min_\bd \hat Q_{\psi_i}\left(\bd\right) \coloneqq \nabla \hat f(\bp^{(i)})^T \bd +
		\frac{\psi_i \left\|\bd\right\|^2}{2}  + \hat g\left(\bd +
		\bp\right) - \hat g\left(\bp\right).
		\label{eq:hatQ}
	\end{equation}
	Denote the optimal solution by $\bd^*$, then we have $\bp^{(i+1)}
	= \bp^{(i)} + \bd^*$.
	Because $H$ is upper-bounded by $c_1$, we have that $\nabla \hat
	f$ is $c_1$-Lipschitz continuous.
	Therefore, using Lemma 12 of \citet{LeeW18a}, we get
	\begin{equation}
		\hat Q_{\psi_i}\left(\bd^*\right) \leq -\frac{\psi_i}{2} \left\|\bd^*\right\|^2.
		\label{eq:Qbound}
	\end{equation}
	We then have from $c_1$-Lipschitz continuity of $\nabla
	\hat f$ that
	\begin{align*}
	&~Q\left(\bp^{(i+1)}\right) - Q\left(\bp^{(i)}\right)\\
	\leq&~ \nabla \hat f(\bp^{(i)})^T \left( \bp^{(i+1)} - \bp^{(i)}
	\right) + \frac{c_1}{2} \left\|\bp^{(i+1)} - \bp^{(i)}\right\|^2
	\\
	&\qquad + \hat g\left(\bp^{(i+1)}\right) - \hat g\left( \bp^{(i)}
	\right) \\
	\stackrel{\eqref{eq:hatQ}}{=}&~ \hat Q_{\psi_i}(\bd^*) -
	\frac{\psi_i}{2} \left\| \bd^* \right\|^2 + \frac{c_1}{2}
	\left\|\bd^*\right\|^2\\
	\stackrel{\eqref{eq:Qbound}}{\leq}&~ \left(\frac{c_1}{2} -
	\psi_i\right)\|\bd^*\|^2.
	\end{align*}
	Therefore, whenever
	\begin{equation*}
		\frac{c_1}{2} - \psi_i \leq -\frac{\sigma_0 \psi_i}{2},
	\end{equation*}
	\eqref{eq:accept} holds.
	This is equivalent to
	\begin{equation*}
		\psi_i \geq \frac{c_1}{2 - \sigma_0}.
	\end{equation*}
	Since $\sigma_0 \in (0,1)$, we must have $c_1 / (2 - \sigma_0) \in
	(c_1 / 2, c_1)$,
	Note that the initialization of $\psi_i$ is upper-bounded by
	$c_1$ for all $i > 1$, so the final $\psi_i$ is upper bounded by
	$2 c_1$.
	Together with the first iteration that we start with $\psi_0 =
	\gamma_t$,
	we have that $\psi_i$ are always upper-bounded by $\max\{ 2 c_1,
	\gamma_t\}$, and we have already proven $\gamma_t$ is
	upper-bounded by $\|X^T X\|L$.

	\item
	We note that since $Q$ is $c_2$-strongly convex,
	the following condition holds.
	\begin{equation}
\frac{\min_{\bs \in \nabla \hat f\left(\bp^{(i + 1)}\right) + \partial \hat
g\left( \bp^{(i+ 1)} \right)}
	\left\|\bs \right\|^2}{2 c_2} \geq Q\left(\bp^{(i + 1)}\right) - Q^*
	\label{eq:kl}
	\end{equation}
	On the other hand, from the optimality condition of
	\eqref{eq:hatQ}, we have that
	\begin{equation}
		\psi_i \bd^* = \nabla \hat f\left(\bp^{(i)}\right) + \bs_{i+1},
		\label{eq:dopt}
	\end{equation}
	for some
	\begin{equation*}
	\bs_{i+1} \in \partial \hat g\left( \bp^{(i+1)}\right).
	\end{equation*}
	Therefore,
	\begin{align}
		\nonumber
		&~Q\left( \bp^{(i+1)} \right) - Q^*\\
		\nonumber
		\stackrel{\eqref{eq:kl}}{\leq}&~ \frac{1}{2 c_2}
		\left\|\nabla \hat f\left( \bp^{(i+1)} \right) - \nabla \hat
		f\left( \bp^{(i)} \right) + \nabla \hat f\left( \bp^{(i)}
		\right) + \bs_{i+1}\right\|^2\\
		\nonumber
		\stackrel{\eqref{eq:dopt}}{\leq}&~ \frac{1}{c_2} \left\|
			\nabla \hat f\left( \bp^{(i+1)} \right)
		- \nabla \hat f\left( \bp^{(i)} \right) \right\|^2 + \left\|\psi_i
		\bd^*\right\|^2\\
		\label{eq:dbound}
		\leq&~ \frac{1}{c_2} \left( c_1^2 + \psi^2 \right)
		\left\|\bd^*\right\|^2.
	\end{align}
	By combining \eqref{eq:accept} and \eqref{eq:dbound}, we obtain
	\begin{align*}
		Q\left( \bp^{(i+1)} \right) - Q\left( \bp^{(i)} \right)
		&\leq -\frac{\sigma_0 \psi_i}{2}\left\|\bd^*\right\|^2\\
		&\leq -\frac{\sigma_0 \psi_i}{2} \frac{c_2}{c_1^2 + \psi^2}
		\left(Q\left( \bp^{(i+1)} \right) - Q^*\right).
	\end{align*}
	Rearranging the terms, we obtain
	\begin{equation*}
		\left(1 + \frac{c_2\sigma_0 \psi_i}{2 ( c_1^2 +
		\psi^2)}\right) \left(Q\left( \bp^{(i+1)} \right) - Q^*\right)
		\leq Q\left( \bp^{(i)} \right) - Q^*,
	\end{equation*}
	showing Q-linear convergence of SpaRSA,
	with
	\begin{equation*}
		\eta = \left( 1 + \frac{c_2 \sigma_0 \psi_i}{2 \left( c_1^2
			+ \psi_i^2
		\right)} \right)^{-1} \in [0,1].
		\qedhere
	\end{equation*}
\end{enumerate}
\end{proof}


\end{document}